\documentclass[12pt]{amsart}
\usepackage[numbers, square]{natbib}
\usepackage{amssymb}
\usepackage{amsmath}
\usepackage{amsfonts}
\usepackage{geometry}
\usepackage{graphicx}
\usepackage{amsthm}
\usepackage{hyperref}
\usepackage[dvipsnames]{xcolor}

\providecommand{\U}[1]{\protect\rule{.1in}{.1in}}
\theoremstyle{plain}

\newtheorem{lemma}{Lemma}

\newtheorem{proposition}{Proposition}

\newtheorem{theorem}{Theorem}
\numberwithin{equation}{section}

\newcommand{\supp}{\mathop{\mathrm{supp}}}
\newcommand{\eps}{\epsilon}

\def\XXint#1#2#3{{\setbox0=\hbox{$#1{#2#3}{\int}$ }
\vcenter{\hbox{$#2#3$ }}\kern-.6\wd0}}

\begin{document}
\title[Geometry and interior nodal sets of Steklov eigenfunctions]{Geometry and interior nodal sets of Steklov
eigenfunctions }
\author{ Jiuyi Zhu}
\address{
 Department of Mathematics\\
Louisiana State University\\
Baton Rouge, LA 70803, USA\\
Emails: zhu@math.lsu.edu}
\thanks{\noindent  Research is partially
supported by the NSF grant DMS 1656845. }
\date{}
 \dedicatory{}

\begin{abstract}
We investigate the geometric properties of Steklov eigenfunctions in
smooth manifolds.  We derive the refined doubling estimates and
Bernstein's inequalities. For the real analytic
manifolds, we are able to obtain the sharp upper bound for the
measure of interior nodal sets.
\end{abstract}

\maketitle
\section{Introduction}
In this paper, we address the geometric properties and interior
nodal sets of Steklov eigenfunctions
\begin{equation}
\left\{
\begin{array}{lll}
\triangle_g e_\lambda(x)=0,
 \quad x\in\mathcal{M},
\medskip \\
\frac{\partial e_\lambda}{\partial \nu}(x)=\lambda e_\lambda(x),
\quad x\in
\partial\mathcal{M},
\end{array}
\right. \label{Stek}
\end{equation}
where $\nu$ is a unit outward normal on $\partial\mathcal{M}$.
 Assume that $(\mathcal{M}, g)$ is a $n$-dimensional
smooth, connected and compact manifold with smooth boundary
$\partial \mathcal{M}$, where $n\geq 2$.  The Steklov eigenfunctions
were first studied by Steklov in 1902 for bounded domains in the
plane. It is also regarded as eigenfunctions of the
Dirichlet-to-Neumann map, which is a first order homogeneous,
self-adjoint and elliptic pseudodifferential operator.
 The spectrum $\lambda_j$ of Steklov eigenvalue problem consists of
 an infinite increasing sequence
 with $$0=\lambda_0<\lambda_1\leq
\lambda_2\leq \lambda_3,\cdots, \
 \ \mbox{and} \ \ \lim_{j\to\infty}\lambda_j=\infty.$$
The eigenfunctions $\{ e_{\lambda_j}\}$ form  an orthonormal basis
such that
$$ e_{\lambda_j}\in C^\infty(\mathcal{M}), \quad
\int_{\partial\mathcal{M}} e_{\lambda_j} e_{\lambda_k}\, dV_{
{g}}=\delta_{j}^{k}.
$$

Recently, the study of nodal geometry of eigenfunctions has been attracting much
attention. Estimating the Hausdorff measure of nodal sets has
always been an important subject concerning the nodal geometry of
eigenfunctions. The celebrated problem about nodal sets centers
around the famous Yau's conjecture for smooth manifolds.
 Let
$e_\lambda$ be $L^2$ normalized eigenfunctions of \begin{equation}
-\triangle_g e_\lambda=\lambda^2 e_\lambda
\label{class}\end{equation} on compact manifolds $(\mathcal{M}, g)$
without boundary,
 Yau conjectured that  the upper and lower bound of nodal sets of
eigenfunctions in (\ref{class}) are controlled by
\begin{equation}c\lambda\leq H^{n-1}(\{x\in\mathcal{M}|e_\lambda(x)=0\})\leq C\lambda \label{yau}
\end{equation} where $C, c$
depend only on the manifold $\mathcal{M}$. The conjecture is shown
to be true for real analytic manifolds by Donnelly-Fefferman in
\cite{DF}, \cite{DF2}. Lin \cite{Lin} also proved the upper bound for the
analytic manifolds using a different approach. 

Let us briefly review the recent literature concerning the progress of Yau's conjecture on nodal sets of classical eigenfunctions (\ref{class}). For the conjecture (\ref{yau}) on the measure of nodal sets on smooth manifolds, there are important breakthrough made by Logunov and Malinnikova \cite{LM},  \cite{Lo1} and \cite{Lo2} in recent years. For the upper bound of nodal sets on two dimensional manifolds,   Logunov and
Malinnikova \cite{LM} showed that $H^{1}(\{x\in\mathcal{M}| e_\lambda(x)=0\})\leq C \lambda^{\frac{3}{2}-\epsilon} $, which slightly improves the upper bound $C \lambda^{\frac{3}{2}}$ by Donnelly and Fefferman \cite{DF1} and  Dong \cite{D}. For higher dimensions $n\geq 3$ on smooth manifolds,  Logunov in \cite{Lo1} obtained a polynomial upper bound.
 The polynomial upper bound improves the exponential upper bound derived by  Hardt and Simon \cite{HS}.
For the lower bound, Logunov \cite{Lo2} completely solved the Yau's conjecture and obtained the sharp lower bound in (\ref{yau}).
For $n=2$, such sharp lower bound was obtained earlier by Br\"uning \cite{Br}. This sharp lower bound  improves a polynomial lower bound obtained early by Colding and  Minicozzi \cite{CM}, Sogge and Zelditch \cite{SZ}, \cite{SZ1}. See also other polynomial lower bounds by different methods, e.g. \cite{HSo}, \cite{M}, \cite{S}.  See also other related results on nodal sets and geometric properties of classical eigenfunctions, e.g. \cite{ChM}, \cite{Han}, \cite{HLu}, \cite{Lu}.
For detailed account about this subject, interested readers may refer to the book
\cite{HL} and survey \cite{Z}.

For the Steklov eigenfunctions, by the maximum principle, there
exist nodal sets in $\mathcal{M}$ and those sets intersect the
boundary $\partial \mathcal{M}$ traservasally.  It is interesting to
ask Yau's type questions about the Hausdorff measure of nodal sets
of Steklov eigenfunctions on the boundary and interior of the manifolds,
respectively. The natural and corresponding conjecture for Steklov
eigenfunctions should state exactly as \begin{equation}
c\lambda\leq H^{n-2}(\{x\in \partial\mathcal {M},
e_\lambda(x)=0\})\leq C\lambda, \label{yau1}\end{equation}
\begin{equation} c\lambda\leq
H^{n-1}(\{x\in \mathcal {M}, e_\lambda(x)=0\})\leq C\lambda.
\label{yau2}\end{equation} See also the survey by Girouard and
Polterovich in \cite{GP} about these open questions.

Recently, much work has been devoted to the bounds of nodal sets of
Steklov eigenfunctions on the  boundary
\begin{equation} Z_\lambda=\{x\in\partial\mathcal{M}|
e_\lambda(x)=0\}. \label{newn}
\end{equation}
 The study of (\ref{newn}) was initiated by Bellova and Lin \cite{BL} who proved the $H^{n-2}(Z_\lambda)\leq
C\lambda^6$ with $C$ depending only on $\mathcal{M}$, if
$\mathcal{M}$ is an analytic manifold. By microlocal analysis
argument, Zelditch \cite{Z1} was able to improve their results and
gave the optimal upper bound $H^{n-2}(Z_\lambda)\leq C\lambda$ for
analytic manifolds. For the smooth manifold $\mathcal{M}$, Wang and
the author in \cite{WZ}  established a lower bound
\begin{equation}H^{n-2}(Z_\lambda)\geq C\lambda^{\frac{4-n}{2}}\label{sogg}\end{equation} by
considering the fact that the Steklov eigenfunctions are
eigenfunctions of first order elliptic pseudodifferential operator.
The polynomial lower bound (\ref{sogg}) is the Steklov analogue of
the lower bounds of nodal sets for classical eigenfunctions
(\ref{class}) obtained in \cite{CM} and \cite{SZ1}.

Concerning about the bounds of interior nodal sets of
eigenfunctions,
$$\mathcal{N}_\lambda=\{x\in\mathcal{M}|e_\lambda(x)=0\},        $$
 Sogge, Wang and the author \cite{SWZ} obtained a
lower bound for interior nodal sets
$$H^{n-1}({\mathcal{N}}_\lambda)\geq
C\lambda^{\frac{2-n}{2}}$$ for a smooth manifold $\mathcal{M}$. The
measure of nodal sets is more clear on surfaces. In \cite{Zh1}, the
author was able to obtain an upper bound for the measure of interior
nodal sets $$H^1({\mathcal{N}}_\lambda)\leq
C\lambda^{\frac{3}{2}}.$$  The singular sets $\mathcal{S}_\lambda=
\{ x\in\mathcal{M}|e_\lambda=0, \nabla e_\lambda=0\} $ are finite
points on the nodal surfaces. It was also shown that
$H^0(\mathcal{S}_\lambda)\leq C\lambda^2$ in \cite{Zh1}. Recently,
Polterovich, Sher and Toth \cite{PST} could verify Yau's type
conjecture for upper and lower bounds in (\ref{yau2}) for the
real-analytic Riemannian surfaces $\mathcal{M}$. Georgiev and Roy-fortin \cite{GR} obtained polynomial upper bounds for interior
nodal sets on smooth manifolds. There are still
many challenges for the study of Steklov eigenfunctions. For
instance, it is well-known that the classical eigenfunctions in
(\ref{class}) are so dense that there are nodal sets in each
geodesic ball with radius $C\lambda^{-1}$. This fundamental result
is crucial to derive the lower bounds of nodal sets for classical eigenfunctions (\ref{class}) in \cite{DF} and
\cite{Br}. For the Steklov eigenfunctions, it is unknown whether
such density results remain true on the boundary and interior of the
manifold, which cause difficulties in studying the Steklov
eigenfunctions.

An interesting topic in the study of eigenfunction is called as the doubling
inequality.  Doubling inequality plays an important role in deriving
strong unique continuation property, the vanishing order of
eigenfunctions and obtaining the measure of nodal sets, see e.g.
\cite{DF}, \cite{DF2}. The doubling inequality for classical
eigenfunctions (\ref{class})
\begin{equation}
\int_{\mathbb B (p,\, 2r)}e_\lambda^2\leq e^{C\lambda} \int_{\mathbb
B(p,\, r)}e_\lambda^2 \label{doub}
\end{equation}
is derived using Carleman estimates in \cite{DF} for $0<r<r_0$,
where $\mathbb B(p, c)$ denotes as a ball in $\mathcal{M}$ centered
at $p$ with radius $c$, and $C$, $r_0$ depends on $\mathcal{M}$.  For
the Steklov eigenfunctions on $\partial\mathcal{M}$, the author has obtained a similar type
of doubling inequality on the boundary $\partial\mathcal{M}$ and
derived that the sharp vanishing order is less than $C\lambda$ on the
boundary $\partial\mathcal{M}$ in \cite{Zh}. For Steklov
eigenfunctions in $\mathcal{M}$, we were also able to get the
doubling inequality as (\ref{doub}) in \cite{Zh1}. For the classical
eigenfunctions (\ref{class}), a refined doubling inequality
\begin{equation}
\int_{\mathbb B (p,\, (1+\frac{1}{\lambda})r)}e_\lambda^2\leq C
\int_{\mathbb B(p,\, r)}e_\lambda^2 \label{rdoub}
\end{equation}
was derived in \cite{DF3} by some  stronger Carleman estimates. The
refine doubling inequality also leads to Bernstein's gradient
inequalities for classical eigenfunctions. The first goal in this note is
to study a refined version doubling inequality for the Steklov
eigenfunctions and its applications.
\begin{theorem}
For the Steklov eigenfunctions in (\ref{Stek}), there hold \\
(A): a refined doubling inequality
\begin{equation}
 \int_{\mathbb B
(p,\, (1+\frac{1}{\lambda})r)}e_\lambda^2\leq C \int_{\mathbb B(p,\,
r)}e_\lambda^2, \nonumber
\end{equation}
(B): $L^2$-Bernstein's inequality
\begin{equation}
\int_{\mathbb B (p,\, r)}|\nabla e_\lambda|^2\leq
\frac{C\lambda^2}{r^2}\int_{\mathbb B(p,\, r)}e^2_\lambda, \nonumber
\end{equation}
(C) $L^\infty$-Bernstein's inequality
\begin{equation}
\max_{\mathbb B (p,\, r)}|\nabla e_\lambda| \leq
\frac{C\lambda^{\frac{n+2}{2}}}{r}\max_{\mathbb B(p,\,
r)}|e_\lambda| \nonumber
\end{equation}
for $\mathbb B (p,\, (1+\frac{1}{\lambda})r)\subset \mathcal{M}$ and
$0<r<r_0$, where $r_0$ depends on $\mathcal{M}$.\label{th1}
\end{theorem}

Our second goal is to obtain the optimal upper bound of interior
nodal sets of Steklov eigenfunctions for real analytic manifolds. Our work extends the optimal
upper bound in \cite{PST} to real analytic manifolds in any dimensions,
which proves the upper bound of Yau's type conjecture for interior
nodal sets in (\ref{yau2}). 
\begin{theorem}
Let $\mathcal{M}$ be a real analytic compact and connected manifold
with boundary. There exists a positive constant $C$ depending on $\mathcal{M}$
such that,
$$  H^{n-1}(\mathcal{N}_\lambda)\leq C\lambda
   $$ for the Steklov eigenfunctions.
   \label{th2}
\end{theorem}
 The outline of the paper is as follows. In section 2, we
reduce the Steklov eigenvalue problem into an equivalent elliptic
equation without boundary. Then we obtain the refined doubling
inequality and show Theorem \ref{th1}. Section 3
is devoted to the upper bound of interior nodal sets for real
analytic manifolds. Section 4 is the appendix which provides the proof of some arguments for the Carleman estimates. The letter $c$, $C$, $C_i$ denote generic
positive constants and do not depend on $\lambda$. They may vary in
different lines and sections. In the paper, since we study the asymptotic properties for eigenfunctions, we assume that $\lambda$ is sufficiently large.

\noindent{\bf Acknowledgement.} It is my pleasure to thank Professor
Christopher D. Sogge and Joel Spruck for support and
helpful discussions about the topic of eigenfunctions. Especially, the author thanks Professor Steve Zelditch for critical comments and constructive suggestions on the nodal sets estimates in section 3.

\section{Refined doubling inequality}

In this section, we will establish a stronger Carleman estimate than
that in \cite{Zh}.  We will
transform the Steklov eigenvalue problem into a second order
elliptic equation on a boundaryless manifold. The eigenvalue
$\lambda$  will be reflected in the coefficient functions of the
elliptic equation.

 To make the Steklov eigenvalue problem into an elliptic equation, adapting the ideas in \cite{BL}, we choose an auxiliary function involving the distance function. Let $d(x)= dist\{x,
\partial\mathcal{M}\}$ be the distance function from $x\in\mathcal{M}$ to the
boundary $\partial\mathcal{M}$. If $\mathcal{M}$ is smooth, $d(x)$
is smooth in the small neighborhood $\mathcal{M}^0_{\rho}$ of
$\partial\mathcal{M}$ in $\mathcal{M}$, where $\mathcal{M}^0_{\rho}=\{x\in \mathcal{M}| dist\{x, \ \partial\mathcal{M}\}\leq \rho\}.$   By the partition of unity, we extend $d(x)$ in a smooth
manner by introducing
\begin{equation}
\varrho(x)=\left \{\begin{array}{lll} d(x) \quad x\in \mathcal{M}^0_{\rho}, \medskip\nonumber\\
l(x) \quad x\in \mathcal{M}\backslash\mathcal{M}^0_{\rho}.
\end{array}
\right. \label{errorr}
\end{equation}
Therefore, the extended function $\varrho(x)$ is a smooth function in
$\mathcal{M}$. We consider an auxiliary function
$$u(x)=e_\lambda \exp\{\lambda \varrho(x)\}.  $$
Then the new function $u(x)$ satisfies
\begin{equation}
\left \{ \begin{array}{lll} \triangle_g u+b(x)\cdot\nabla_g
u+q(x)u=0 &\quad \mbox{in}\ \mathcal{M},
\medskip\\
\frac{\partial u}{\partial \nu}=0 &\quad \mbox{on}\
\partial\mathcal{M}
\end{array}
\right.
\end{equation}
with
\begin{equation}
\left \{\begin{array}{lll}
b(x)=-2\lambda \nabla_g \varrho(x),   \medskip\\
q(x)=\lambda^2|\nabla_g \varrho(x)|^2-\lambda\triangle_g \varrho(x).
\end{array}
\right. \label{fffw}
\end{equation}
In order to construct a boundaryless model, we attach two copies of
$\mathcal{M}$ along the boundary and consider a double manifold
$\overline{\mathcal{M}}=\mathcal{M}\cup \mathcal{M}$. Then induced
metric ${g'}$ of ${g}$  on the double manifold
$\overline{\mathcal{M}}$ is Lipschitz. We consider a canonical
involutive isometry $\mathcal {F}: \overline{\mathcal{M}}\to
\overline{\mathcal{M}}$ which interchanges the two copies of
${\mathcal{M}}$. In this sense, the function $u(x)$ can be extended
to the double manifold by a even extension as $\overline{\mathcal{M}}$ by $u\circ \mathcal
{F}= u$. Thus, $u(x)$ satisfies
\begin{equation}
\triangle_{g'} u+\bar{b}(x)\cdot\nabla_{g'} u+\bar{ q} (x)u=0 \quad
\mbox{in}\ \overline{\mathcal{M}}.  \label{star}
\end{equation}
Note that the new metric $g'$ is Lipschitz metric.
From the assumptions in (\ref{fffw}) and the even extension, it follows that
\begin{equation}
\left \{\begin{array}{lll}
\|\bar b\|_{W^{1, \infty}(\overline{\mathcal{M}})}\leq C\lambda, \medskip\\
\|\bar q\|_{W^{1, \infty}(\overline{\mathcal{M}})}\leq C\lambda^2.
\label{core}
\end{array}
\right.
\end{equation}

 By a standard regularity
argument for dealing with Lipschitz metrics in \cite{DF2} and \cite{AKS}, we have
established some quantitative Carleman inequality  in \cite{Zh}
 for the general second order elliptic
equation (\ref{star}). See also e.g. \cite{BC} for similar estimates
for smooth manifolds. The quantitative Carleman estimate inequality is stated as follows.

\begin{lemma}
There exists positive constants $\epsilon_0$ and $C$ such that for 
any $u\in C^{\infty}_{0}\Big(\mathbb B(p, \epsilon_0)\backslash
\mathbb B(p, \epsilon_1)\Big)$, and  $\beta>C(1+\|\bar b\|_{W^{1,
\infty}}+\|\bar q\|^{1/2}_{W^{1, \infty}})$, one has
\begin{equation}
\int r^4 e^{2\beta \psi(r)}|\triangle_{g'} u+\bar b \cdot
\nabla_{g'} u+ \bar q u|^2\, dvol\geq C \beta^3\int r^{\eps}
e^{2\beta\psi(r)} u^2 \,dvol, \label{cca}
\end{equation}
where $\psi(r)=-\ln r(x)+r^\eps(x)$, $r(x)$ is the geodesic
distance from $x$ to $p$, and $0<\eps<1$ is some fixed
constant. \label{carl}
\end{lemma}
Since some arguments are used in the  proof of Proposition \ref{pro2} later, we include the major arguments of the proof of Lemma \ref{carl} in the appendix,
By the Carleman estimates in Lemma \ref{carl}, we can derive a Hadamard's
three-ball theorem. Based on a propagation of smallness argument,
 we have obtained
 the following doubling inequality in $\overline{\mathcal{M}}$ in \cite{Zh1}.
\begin{proposition}
There exist positive constants $r_0$ and $C$  depending only on
$\overline{\mathcal{M}}$ such that for any $0<r<r_0$ and any $p\in
\overline{\mathcal{M}}$, there holds
\begin{align}
\|u\|_{L^2(\mathbb B(p, \,{2r}))}\leq e^{C\lambda}\|u\|_{L^2(\mathbb
B(p,\, {r}))} \label{ddou}
\end{align}
for any solutions of (\ref{star}). \label{proo}
\end{proposition}
From the proposition, it is easy to see that the doubling inequality
for Steklov eigenfunctions as (\ref{doub}) holds in $\mathcal{M}$ if
$\mathbb B(p, {2r})\subset\mathcal{M}$, since $\varrho(x)$ is a bounded function. By standard elliptic
estimates, the $L^\infty$ norm of doubling inequality
\begin{align}
\|u\|_{L^\infty(\mathbb B(p, \,{2r}))}\leq
e^{C\lambda}\|u\|_{L^\infty(\mathbb B(p,\, {r}))} \label{kaoni}
\end{align}
holds, which also implies that
$$
\|e_\lambda\|_{L^\infty(\mathbb B(p, \,{2r}))}\leq
e^{C\lambda}\|e_\lambda\|_{L^\infty(\mathbb B(p,\, {r}))}.
$$
Next we will establish a stronger Carleman inequality than that in
Lemma \ref{carl} with weight function $\exp\{\beta \psi(x)\}$ following from \cite{DF3}, where
the function $\psi$ satisfies some convexity properties. Choosing a
fixed number $\eps$ such that $0<\eps<1$ and $T_0<0$, we define the
function $\phi$ on $(-\infty, \ T_0]$ by $\phi(t)=t- e^{\eps t}$. If
$|T_0|$ is sufficiently large, the function $\phi(t)$ satisfies the
following properties
\begin{eqnarray}
1-\eps e^{\eps T_0}\leq \phi'(t)\leq 1, \label{ass1}\\
\lim\limits_{t\to-\infty}\frac{-\phi''(t)}{e^t}=+\infty.
\label{ass2}
\end{eqnarray}
Let $\psi(x)=-\phi\big(\ln r(x)\big)$, where $r(x)=d(x,\, p)$ is
geodesic distance between $x$ and $p$. The stronger Carleman estimate is stated as follows.

\begin{proposition} There exist positive constants $h$, $C_0$ and $C$ such that for
any $u\in C^{\infty}_{0}\Big(\mathbb B(p, h)\backslash \mathbb B(p,
\delta)\Big)$, and $\beta>C_0(1+\|\bar b\|_{W^{1, \infty}}+\|\bar
q\|^{1/2}_{W^{1, \infty}})$, one has
\begin{eqnarray}
\int_{\mathbb \mathbb B(p, h)} r^4 e^{2\beta \psi}|\triangle u+\bar
b\cdot \nabla u+\bar q u|^2 \, dvol &\geq& C \beta^3 \int_{\mathbb
B(p, h)} r^\epsilon e^{2\beta\psi} u^2 dvol \nonumber \\&+&C\beta^4
\int_{\mathbb B(p, \delta(1+\frac{C}{\beta}))} e^{2\beta \psi} u^2
dvol. \label{comb1}
\end{eqnarray}
 \label{pro2}
\end{proposition}
\begin{proof} By the standard argument in dealing with Lipschitz Riemannian manifold in \cite{DF2} and \cite{AKS},
using a conformal change, we can still use  polar geodesic coordinates $ (r, \omega)$. The change only results in the change of $C$ in the norm estimates of coefficient functions in (\ref{core}). The geodesic balls are comparable under the conformal change of the different metrics. For simplicity, we still keep the notations in (\ref{star}).
We introduce the polar geodesic coordinates $ (r, \omega)$
near $p$. Following the Einstein notation, for any $v\in C^\infty$, we denote the
Laplace-Beltrami operator as
$$ r^2 \triangle v=r^2 \partial^2_r v+ r^2\big( \partial_r \ln(\sqrt{\gamma})+\frac{n-1}{r}\big) \partial_r v+
\frac{1}{\sqrt{\gamma}}
\partial_i\big(\sqrt{\gamma}\gamma^{ij}\partial_j v\big),
$$
where $\partial_i =\frac{\partial}{\partial \omega_i}$ and $
\gamma_{ij}(r, \omega)$ is a metric on $S^{n-1}$,
$\gamma=\det(\gamma_{ij})$. One can check that, for $r$ small
enough,
\begin{equation}
\left \{ \begin{array}{lll}
\partial_r(\gamma_{ij})\leq C (\gamma_{ij}) \ \ \mbox{in term of
tensors},\medskip \\
|\partial_r(\gamma)|\leq C, \label{gamma} \medskip\\
C^{-1} \leq \gamma \leq C,
\end{array}
\right.
\end{equation}
where $C$ depends on $\mathcal{M}$.
Set a new coordinate as $\ln r=t$. Using this new
coordinate,
\begin{equation} e^{2t} \triangle v= \partial^2_t v+ ( n-2+ \partial_t\ln\sqrt{\gamma}) \partial_t v+
\frac{1}{\sqrt{\gamma}}
\partial_i\big(\sqrt{\gamma}\gamma^{ij}\partial_j v\big)
\label{new}
\end{equation}
and
$$ e^{2t} \bar b=e^{2t}\bar b_t \partial_t+ e^{2t} \bar b_i \partial_i.          $$
Since $u$ is supported in a small neighborhood, then $u$ is
supported in $(-\infty, \ T_0)\times S^{n-1}$ with $T_0<0$ and $|T_0|$ large
enough. Under this new coordinate, the condition (\ref{gamma})
becomes
\begin{equation}
\left \{ \begin{array}{lll}
\partial_t(\gamma_{ij})\leq C e^t (\gamma_{ij}) \ \ \mbox{in term of
tensors},\medskip \\
|\partial_t(\gamma)|\leq C e^t, \label{gamma1} \medskip\\
C^{-1} \leq \gamma \leq C.
\end{array}
\right.
\end{equation}
Let $$u=e^{-\beta \psi(x)} v.$$ Define the conjugate operator,
\begin{align}\mathcal {L}_\beta (v)&=r^2  e^{\beta\psi(x)} \triangle (e^{-\beta \psi(x)}v)+ r^2 e^{\beta \psi(x)} \bar
b\cdot \nabla (e^{-\beta \psi(x) }v )+  r^2\bar q v  \nonumber \\
&=e^{2t} e^{-\beta\phi(t)} \triangle (e^{\beta \phi(t)}v)+e^{2t} e^{-\beta \phi(t)} \bar
b\cdot \nabla (e^{\beta \phi(t) }v ) +e^{2t} \bar q v.
\end{align}
From (\ref{new}), straightforward calculations show that \begin{eqnarray}\mathcal
{L}_\beta (v)&=&\partial^2_t v+\big(2\beta \phi'+e^{2t} {\bar
b}_t+(n-2)+\partial_t \ln\sqrt{\gamma}\big)\partial_t v+ e^{2t} \bar
b_i\partial_i v \nonumber \\
&+&\big(\beta^2\phi'^2+\beta\phi'{\bar b}_t e^{2t}+\beta
\phi''+(n-2)\beta \phi'+\beta \partial_t \ln \sqrt{\gamma}\phi'\big)
v+\triangle_{\omega} v+ e^{2t} \bar q v,
\end{eqnarray}
where $$\triangle_{\omega} v=\frac{1}{\sqrt{\gamma}}\partial_i
(\sqrt{\gamma} \gamma^{ij}\partial_j v).$$ We will work in the
following $L^2$ norm
$$ \|v\|^2_{\phi}=\int_{(-\infty, \ T_0]\times S^{n-1}} |v|^2\sqrt{\gamma} \phi'^{-3} \, dt d\omega,           $$
where $d\omega$ is measure on $S^{n-1}$. By the triangle inequality, we have
$$ \|\mathcal{L}_\beta (v)\|^2_\phi\geq \frac{1}{2}\mathcal{A}-\mathcal{B},              $$
where \begin{eqnarray} \mathcal{A}=\|
\partial^2_t v&+& \triangle_{\omega} v +\big(2\beta \phi'+e^{2t} {\bar
b}_t\big)\partial_t v+ e^{2t} \bar
b_i\partial_i v \nonumber \\
&+&\big(\beta^2\phi'^2+\beta\phi'{\bar b}_t e^{2t}+(n-2)\beta \phi'+
e^{2t} \bar q \big) v\|^2_{\phi} \end{eqnarray} and \begin{equation}
\mathcal{B}=\|\beta \phi'' v+\beta \partial_t \ln\sqrt{\gamma} \phi'
v+(n-2)\partial_t v+\partial_t \ln\sqrt{\gamma}\partial_t
v\|^2_{\phi}.
\end{equation}
By integration by parts argument, we can absorb
$\mathcal{B}$ into $\mathcal{A}$. It holds that
\begin{equation}
\|\mathcal{L}_\beta (v)\|^2_\phi\geq
\frac{1}{4}\mathcal{A}. \label{four}
\end{equation}
We can also obtain a lower bound for $\mathcal{A}$,
\begin{eqnarray}
C\mathcal{A}&\geq& \beta^3\int |\phi''||v|^2
\phi'^{-3}\sqrt{\gamma}\,dt d \omega +\beta \int |\phi''||D_\omega
v|^2
\phi'^{-3}\sqrt{\gamma}\,dt d \omega \nonumber \\
&&+\beta\int |\partial_t v|^2 \phi'^{-3}\sqrt{\gamma}\,dt d \omega,
\label{doit}
\end{eqnarray}
where $|D_{\omega} v|^2$ stands for
$$|D_{\omega} v|^2=\gamma^{ij}\partial_i v\partial_j v.   $$
For the completeness of the presentation, we include the proof of (\ref{four}) and (\ref{doit}) in the Appendix.

We also want to find another
refined lower bound for $\mathcal{A}$. We write $\mathcal{A}$ as
\begin{align} \mathcal{A}=\mathcal{A}_1+\mathcal{A}_2+\mathcal{A}_3+\mathcal{A}_4,    \label{fuckz}  \end{align}
where $$\mathcal{A}_1=\|\partial^2_t v
+\big(\beta^2\phi'^2+\beta\phi'{\bar b}_t e^{2t}+(n-2)\beta \phi'+
e^{2t} \bar q \big) v+ \triangle_{\omega} v\|^2_{\phi} $$
and
$$\mathcal{A}_2=\|\big(2\beta \phi'+e^{2t} {\bar
b}_t\big)\partial_t v+e^{2t}\bar{b}_i\partial_i v+\beta g
v\|^2_{\phi}
$$
and
\begin{eqnarray}\mathcal{A}_3=2< \partial^2_t v+
\big(\beta^2\phi'^2+\beta\phi'{\bar b}_t e^{2t}+(n-2)\beta \phi'+
e^{2t} \bar q \big) v+\triangle_\omega v-\beta g v,\nonumber \\
\big(2\beta \phi'+e^{2t} {\bar b}_t\big)\partial_t
v+e^{2t}\bar{b}_i\partial_i v>_\phi \nonumber
\end{eqnarray}
and
$$\mathcal{A}_4=-\beta^2 \|gv\|^2_{\phi},
$$
 and $g(t)$ is a function to be determined. We continue to break
$\mathcal{A}_3$ down as\begin{equation}\mathcal{A}_3=\mathcal
{I}_1+\mathcal{I}_2, \label{brea}\end{equation}where
\begin{eqnarray}\mathcal {I}_1=2<\partial^2_t v+
\big(\beta^2\phi'^2+\beta\phi'{\bar b}_t e^{2t}+(n-2)\beta \phi'+
e^{2t} \bar q \big) v+\triangle_\omega v,
\nonumber \\
\big(2\beta \phi'+e^{2t} {\bar b}_t\big)\partial_t
v+e^{2t}\bar{b}_i\partial_i v>_\phi \nonumber
\end{eqnarray}
and
$$ \mathcal{I}_2=-2<\beta g v, \   \big(2\beta \phi'+e^{2t} {\bar b}_t\big)\partial_t
v+e^{2t}\bar{b}_i\partial_i v>_\phi. $$ Performing the integration by
part arguments  shows that
\begin{eqnarray}
\mathcal {I}_1&\geq& 3\beta\int |\phi''||D_{\omega} v|^2
\phi'^{-3}\sqrt{\gamma}\,dt d \omega -c\beta^3\int e^{t}|v|^2
\phi'^{-3}\sqrt{\gamma}\,dt d \omega \nonumber \\
&&-c\beta\int |\phi''||\partial_t v|^2 \phi'^{-3}\sqrt{\gamma}\,dt d \omega
-c\beta^2\int |\phi''|| v|^2 \phi'^{-3}\sqrt{\gamma}\,dt d \omega.
\label{kao1}
\end{eqnarray}

 From (\ref{doit})  and (\ref{kao1}), it follows that
 $$\mathcal{I}_1+C'\mathcal{A}\geq 0  $$
 for some positive constant $C'$. That is,
\begin{equation}
\mathcal{I}_1\geq -C'\mathcal{A}. \label{kao}
\end{equation}

 We compute $\mathcal{I}_2$. Applying the integrating by parts gives that
\begin{align}
\mathcal{I}_2 &=\int \beta g'(2\beta \phi'+e^{2t} {\bar b}_t) v^2
 \phi'^{-3}\sqrt{\gamma} \, dt d\omega \nonumber \medskip \\&+\int \beta g(2\beta \phi''+2e^{2t} {\bar b}_t+e^{2t}\partial_t{\bar
b}_t) v^2 \phi'^{-3}\sqrt{\gamma} \, dt d\omega\nonumber \medskip \\
&-\int 3 \beta g(2\beta \phi'+ e^{2t} {\bar b}_t){\phi'}^{-1}
\phi'' v^2\phi'^{-3}\sqrt{\gamma} \, dt d\omega \nonumber \medskip
\\&+\int \beta g(2\beta
\phi'+e^{2t} {\bar b}_t)\partial_t\ln\sqrt{\gamma}v^2\phi'^{-3}\sqrt{\gamma} \, dt d\omega \nonumber \\
\medskip &+\int \beta g e^{2t}(\partial_i \bar{b}_i +
\bar{b}_i
\partial_i \ln\sqrt{\gamma})v^2\phi'^{-3}\sqrt{\gamma} \, dt
d\omega. \nonumber \end{align} Combining terms in the later
identity yields that
\begin{align}
\mathcal{I}_2 &=\beta^2 \int \Big\{ g'(2\phi'+\frac{e^{2t}{\bar
b}_t}{\beta})+g\Big(2 \phi''+ \frac{ 2e^{2t}{\bar
b}_t+e^{2t}\partial_t \bar{b}_t}{\beta} -6
\phi''-3\frac{e^{2t}\bar{b}_t {\phi'}^{-1} \phi''}{\beta} \nonumber
\medskip\\ &+ (2\phi'+\frac{e^{2t}\bar
b_t}{\beta})\partial_t\ln\sqrt{\gamma} + \frac{ e^{2t}\partial_i
\bar{b}_i+e^{2t}\bar{b}_i\partial_i\ln\sqrt{\gamma}}{\beta}\Big)\Big\}
v^2 \phi'^{-3} \sqrt{\gamma}\, dt d\omega.\label{nie}
\end{align}
Since $\mathcal{A}_1$ and $\mathcal{A}_2$ are nonnegative, from (\ref{fuckz}),
(\ref{brea}) and (\ref{kao}), we have
\begin{align*}
\mathcal{A}&\geq \mathcal{I}_1+\mathcal{I}_2+ \mathcal{A}_4 \nonumber \\
&\geq -C'\mathcal{A}+\mathcal{I}_2+ \mathcal{A}_4.
\end{align*}
By (\ref{nie}), we have a lower bound of $\mathcal{A}$ as
\begin{eqnarray}
C\mathcal{A}&\geq& \beta^2 \int \Big\{\Big[ g'(2\phi'+\frac{e^{2t}{\bar b}_t}{\beta})+g\Big( \frac{ 2e^{2t}{\bar b}_t+e^{2t}\partial_t \bar{b}_t}{\beta}
 -4\phi''-3\frac{e^{2t}\bar{b}_t {\phi'}^{-1}
\phi''}{\beta}\nonumber \\&& +(2\phi'+\frac{e^{2t}\bar
b_t}{\beta})\partial_t\ln\sqrt{\gamma}+ \frac{ e^{2t}\partial_i
\bar{b}_i+e^{2t}\bar{b}_i\partial_i\ln\sqrt{\gamma}}{\beta}\Big)\Big]-g^2\Big\}
v^2 \phi'^{-3} \sqrt{\gamma}\, dt d\omega. \label{lead}
\end{eqnarray}
From the assumption (\ref{ass1}), we know $\phi'$ is close to $1$ as
$|T_0|$ is sufficiently large. By the assumption of $\bar b$ and the condition $\beta>C(1+\|\bar b\|_{W^{1, \infty}}+\|\bar
q\|^{1/2}_{W^{1, \infty}})$, it is
clear that $|\frac{e^{2t}\bar b_t}{\beta}|$ is small. Thus, the
condition
$$ 2\phi'+  \frac{e^{2t}\bar b_t}{\beta}>0        $$
holds. Let
\begin{eqnarray}
g'(2\phi'+\frac{e^{2t}{\bar b}_t}{\beta})+g\Big(\frac{ 2e^{2t}{\bar
b}_t+e^{2t}\partial_t \bar{b}_t}{\beta} -4
\phi''-3\frac{e^{2t}\bar{b}_t {\phi'}^{-1} \phi''}{\beta}\nonumber
+(2\phi'+\frac{e^{2t}\bar b_t}{\beta})\partial_t\ln\sqrt{\gamma}\\
+\frac{ e^{2t}\partial_i
\bar{b}_i+e^{2t}\bar{b}_i\partial_t\ln\sqrt{\gamma}}{\beta}\Big)-g^2=\beta^2
(2\phi'+\frac{e^{2t}{\bar b}_t}{\beta})\varphi(\beta (t-t_\ast)),
\label{ast}
\end{eqnarray}
where $\varphi(t)=0$ for $t\geq 0$, $\varphi(t)>0$ for $t<0$, and
$|t_\ast|$ is an arbitrary large number with $t_\ast<0$. We attempt to solve
(\ref{ast}) with $g=0$ for $t\geq t_\ast$. Making the change of
rescale, we have
$$ g=\beta G, \quad \quad z=\beta(t-t_\ast).     $$
Then (\ref{ast}) is transformed into an equation of the form
\begin{equation}
\left\{ \begin{array}{lll}
\frac{\partial G}{\partial z}= H_1(z)+ H_2(z) G+ H_3(z)G^2, \nonumber \medskip \\
G(0)=0,
\end{array}
\right.
\end{equation}
with $H_1, H_2$ and $H_3$ are uniformly bounded in $C^2$. Standard existence
theorem from ordinary differential equations shows a solution to
(\ref{ast}) for $-C_1\leq \beta (t-t_\ast)\leq 0$ with a fixed small
positive constant $C_1$. Then (\ref{ast}) can be solved for
$\frac{-C_1}{\beta}+t_\ast\leq t\leq t_\ast$. If we assume that
$\supp v\subset \{\frac{-C_1}{\beta}+t_\ast\leq t\leq T_0\} $ with
$T_0<0$, then (\ref{lead}) implies that
\begin{equation}
 C\mathcal{A}\geq \beta^4 \int (2\phi'+\frac{e^{2t}{\bar
b}_t}{\beta})\varphi(\beta (t-t_\ast))  v^2 \phi'^{-3}\sqrt{\gamma}
\, dt d\omega.
\end{equation}
There exist $0<-T_0<C_2<C_3<C_1$ such that
$$ \varphi(z)> C_4 \quad \quad \mbox{for} \ -C_3<z<-C_2       $$
and $C_4$ depends on $C_2$, $C_3$. It follows from the last
inequality that
\begin{equation}
C \mathcal{A}\geq C_4 \beta^4 \int_{
 t_\ast-\frac{C_3}{\beta}<t<t_\ast-\frac{C_2}{\beta}}  v^2 \phi'^{-3}\sqrt{\gamma} \, dt
d\omega.
\end{equation}
Since $r= e^{t}$ and recall that $u=e^{-\beta \psi(x)}v$, the
previous estimates yield that
\begin{equation}
 \mathcal{A}\geq C_5 \beta^4 \int_{
 t_\ast-\frac{C_3}{\beta}<\ln r <t_\ast-\frac{C_2}{\beta}}e^{2\beta \psi(x)} u^2 r^{-1}\phi'^{-3} \sqrt{\gamma}\,
 dr
d\omega.
\end{equation}
Set $e^{t_\ast}=r_\ast$. If $r_\ast \exp\{-\frac{C_3}{\beta}\}<r<
r_\ast \exp\{-\frac{C_2}{\beta}\}$, there exist positive constants
$C_6$ and  $C_7$ such that
$r_\ast(1-\frac{C_6}{\beta})<r<r_\ast(1-\frac{C_7}{\beta})$. Recall the estimates
(\ref{four}), it follows that
\begin{equation}
\|\mathcal{L}_\beta (v)\|^2_\phi\geq C_5
\beta^4\int_{r_\ast(1-\frac{C_6}{\beta})<r<r_\ast(1-\frac{C_7}{\beta})}
e^{2\beta \psi(x)} u^2 r^{-1}\phi'^{-3}\sqrt{\gamma}dr d\omega.
\end{equation}
Note that $\phi'$ is close to 1, we have
\begin{equation}
\|\mathcal{L}_\beta (v)\|^2\geq C_5 \beta^4
\int_{r_\ast(1-\frac{C_6}{\beta})<r<r_\ast(1-\frac{C_7}{\beta})}
e^{2\beta \psi(x)} u^2 \, d vol
\end{equation}
by a constant change of the value of $\beta$. Since $u\in
C^\infty_0\big(\mathbb B(p, h) \backslash
 \mathbb B(p, \delta)\big)$, choosing
$r_\ast=\frac{\delta}{1-\frac{C_6}{\beta}}$, we have
\begin{equation}
\|r^2 e^{\beta \psi}|\triangle u+\bar
b\cdot \nabla u+\bar q u|\|^2  \geq C_5 \beta^4
\int_{\delta<r<\delta(1+\frac{C_8}{\beta})} e^{2\beta \psi(x)} u^2
\, d vol. \label{how1}
\end{equation}
From Lemma \ref{carl}, we have established that
\begin{equation}
\|r^2 e^{\beta \psi}|\triangle u+\bar
b\cdot \nabla u+\bar q u\| \geq C_9 \beta^{\frac{3}{2}} \|
r^{\frac{\epsilon}{2}}e^{\beta\psi}u\|. \label{how2}
\end{equation}
 Combining those two Carleman
inequalities (\ref{how1}) and (\ref{how2}) yields that
\begin{eqnarray}
\int_{\mathbb \mathbb B(p, h)} r^4 e^{2\beta \psi}|\triangle u+\bar
b\cdot \nabla u+\bar q u|^2 \, dvol &\geq& C \beta^3 \int_{\mathbb
B(p, h)} r^\epsilon e^{2\beta\psi} u^2 d vol \nonumber \\&+&C\beta^4
\int_{\mathbb B(p, \delta(1+\frac{C_8}{\beta}))} e^{2\beta \psi} u^2
dvol \label{comb}
\end{eqnarray}
for $u\in C^\infty_0\big(\mathbb B(p, h) \backslash \mathbb B(p,
\delta)\big)$ and $\beta>C(1+\|\bar b\|_{W^{1, \infty}}+\|\bar
q\|_{W^{1, \infty}})$.
\end{proof}

With aid of the Carleman estimates (\ref{comb}), we are in the
position to give the proof of Theorem \ref{th1}. The refined
doubling inequality and Bernstein's inequalities have been obtained
for classical eigenfunctions in \cite{DF3}.
\begin{proof}[Proof of Theorem \ref{th1}]

 We introduce a cut-off
function $\theta(x)\in C^\infty_0\big(\mathbb B(p, h) \backslash
\mathbb \mathbb B(p, \delta)\big)$ satisfying the
following properties: \\
(i): $\theta=1$ in $\mathbb B(p, \frac{h}{2})\backslash \mathbb B(p,
\delta+\frac{C\delta}{10\beta})$, \medskip \\
(ii): $|\nabla \theta|\leq \frac{C\beta}{\delta}$, $|\triangle
\theta|\leq \frac{C\beta^2}{\delta^2}$ in $\mathbb B(p,
\delta+\frac{C\delta}{10\beta}),$\medskip \\
(iii): $|\nabla \theta|\leq C$ in $\mathbb B(p, h)\backslash \mathbb
B(p,
\frac{h}{2})$. \\
Let $w(x)=\theta(x) u(x)$.  Since $u$ satisfies
$$\triangle u+\bar b\cdot \nabla u+\bar q u=0,$$ then $w$ satisfies
$$\triangle w+\bar b\cdot \nabla w+\bar q w=\triangle \theta u+2 \nabla \theta \cdot \nabla
u+ \bar b\cdot\nabla\theta u.           $$ Substituting $w$ into the
left hand side of the stronger inequality (\ref{comb})
 and calculating its integrals gives that
\begin{eqnarray}
&&\int_{\big(\mathbb B(p, h)\backslash \mathbb B(p,
\frac{h}{2})\big)\bigcup \big(\mathbb B(p,
\delta+\frac{C\delta}{10\beta})\backslash \mathbb B(p, \delta)\big)}
r^4 e^{2\beta\psi}|\triangle \theta u+2 \nabla \theta \cdot \nabla
u+\bar b  \cdot \nabla\theta u|^2 \nonumber\\
&&\leq  C\beta^2\int_{\mathbb B(p, h)\backslash \mathbb B(p,
\frac{h}{2})} r^4 e^{2\beta\psi}(u^2+|\nabla u|^2) \nonumber\\ &&+C\int_{\mathbb
B(p, \delta+\frac{C\delta}{10\beta})\backslash \mathbb B(p, \delta)}
r^4
e^{2\beta\psi}(\frac{\beta^4}{\delta^4}u^2+\frac{\beta^2}{\delta^2}
|\nabla u|^2+\frac{\beta^4}{\delta^2}u^2), \nonumber
\end{eqnarray}
where we have used the assumption for $\bar b$ and $\bar q$ in
(\ref{core}) and the assumption $\beta>C(1+\|\bar b\|_{W^{1,
\infty}}+\|\bar q\|^{1/2}_{W^{1, \infty}})$.

Substituting $w$ into the right hand side of (\ref{comb}) and taking the later inequality
into consideration yields that
\begin{eqnarray}
C\int_{\mathbb B(p, h)\backslash \mathbb B(p, \frac{h}{2})} r^4
e^{2\beta\psi}(u^2+|\nabla u|^2)+C\int_{\mathbb B(p,
\delta+\frac{C\delta}{10\beta})\backslash \mathbb B(p, \delta)} r^4
e^{2\beta\psi}(\frac{\beta^2}{\delta^4}u^2+\frac{1}{\delta^2}
|\nabla u|^2)\nonumber \\
\geq \beta \int_{\mathbb B(p, \frac{h}{2})\backslash \mathbb B(p,
\delta+\frac{C\delta}{10\beta})} r^\epsilon
e^{2\beta\psi}u^2+\beta^2 \int_{\mathbb B(p,
\delta+\frac{C\delta}{\beta})\backslash \mathbb B(p,
\delta+\frac{C\delta}{10\beta})}  e^{2\beta\psi}u^2. \label{late}
\end{eqnarray}
Using the fact that $\psi$ is a decreasing function and the standard elliptic
estimates, we have
\begin{eqnarray}
\int_{\mathbb B(p, h)\backslash \mathbb B(p, \frac{h}{2})} r^4
e^{2\beta\psi}|\nabla u|^2 &\leq& Ch^4 e^{2\beta\psi(\frac{h}{2})}
\int_{\mathbb B(p, h)\backslash \mathbb B(p, \frac{h}{2})}|\nabla
u|^2 \nonumber \\
&\leq & C\lambda^2 h^2
e^{2\beta\psi(\frac{h}{2})}\int_{\mathbb B(p, \frac{5h}{4})\backslash \mathbb B(p, \frac{h}{4})}u^2.
\label{mor1}
\end{eqnarray}
Thus, 
\begin{eqnarray}
C\int_{\mathbb B(p, h)\backslash \mathbb B(p, \frac{h}{2})} r^4
e^{2\beta\psi}(u^2+|\nabla u|^2)\leq C\lambda^2 h^2
e^{2\beta\psi(\frac{h}{2})}\int_{\mathbb B(p, \frac{5h}{4})\backslash \mathbb B(p, \frac{h}{4})}u^2.
\label{mor1}
\end{eqnarray}
By the decreasing property of $\psi$, we have 
\begin{eqnarray}
\beta \int_{\mathbb B(p, \frac{h}{2})\backslash \mathbb B(p,
\delta+\frac{C\delta}{10\beta})} r^\epsilon e^{2\beta\psi}u^2 &\geq
& \beta \int_{\mathbb B(p, \frac{h}{10})\backslash \mathbb B(p, \frac{h}{20})} r^\epsilon e^{2\beta\psi} u^2 \nonumber \\
 &\geq
& \beta h^\epsilon e^{2\beta\psi(\frac{h}{10})} \int_{\mathbb B(p, \frac{h}{10})\backslash \mathbb B(p, \frac{h}{20})} u^2.
 \label{mor2}
\end{eqnarray}
From the doubling inequality in \cite{Zh1}, we learn that
\begin{align}
e^{C{\lambda}} \int_{\mathbb B(p, \frac{h}{10})\backslash \mathbb B(p, \frac{h}{20})} u^2 \geq \int_{\mathbb B(p, \frac{5h}{4})\backslash \mathbb B(p, \frac{h}{4})}u^2
\end{align}
for some $C$ depending on $\mathcal{M}$. If we choose $\beta>{C}_0{\lambda}$ for some large constant ${C}_0$, from (\ref{mor1}) and (\ref{mor2}), we arrive at
\begin{equation}
\beta \int_{\mathbb B(p, \frac{h}{2})\backslash \mathbb B(p,
\delta+\frac{C\delta}{10\beta})} r^\epsilon e^{2\beta\psi}u^2 \geq
 C \int_{\mathbb
B(p, h)\backslash \mathbb B(p, \frac{h}{2})} r^4
e^{2\beta\psi}(|\nabla u|^2+u^2). \label{fur}
\end{equation}
The combination of (\ref{late}) and (\ref{fur}) yields that
\begin{equation}
\int_{\mathbb B(p, \delta+\frac{C\delta}{10\beta})\backslash \mathbb
B(p, \delta)} r^4
e^{2\beta\psi}(\frac{\beta^2}{\delta^4}u^2+\frac{1}{\delta^2}
|\nabla u|^2)\geq \beta^2 \int_{\mathbb B(p,
\delta+\frac{C\delta}{\beta})\backslash \mathbb B(p,
\delta+\frac{C\delta}{10\beta})}  e^{2\beta\psi}u^2.
\end{equation}
We continue to simplify the last inequality,
\begin{eqnarray}
(\delta+\frac{C\delta}{10\beta})^4
e^{2\beta\psi(\delta)}\frac{\beta^2}{\delta^4}\int_{\mathbb B(p,
\delta+\frac{C\delta}{10\beta})\backslash \mathbb B(p, \delta)} u^2
&+&(\delta+\frac{C\delta}{10\beta})^4
e^{2\beta\psi(\delta)}\frac{1}{\delta^2} \int_{\mathbb B(p,
\delta+\frac{C\delta}{10\beta})\backslash \mathbb B(p,
\delta)}|\nabla u|^2 \nonumber \\
&\geq & \beta^2
e^{2\beta\psi(\delta+\frac{C\delta}{\beta})}\int_{\mathbb B(p,
\delta+\frac{C\delta}{\beta})\backslash \mathbb B(p,
\delta+\frac{C\delta}{10\beta})} u^2.
\end{eqnarray}
From the explicit form of $\psi(x)$, there exists some small
positive constant $c$ such that
$$ \exp\{ 2\beta\psi(\delta+\frac{C\delta}{\beta})- 2\beta\psi(\delta)\}>
c$$  for $\beta$ large enough. Thus,
\begin{equation}
\frac{\beta^2}{\delta^2}\int_{\mathbb B(p,
\delta+\frac{C\delta}{10\beta})\backslash \mathbb B(p, \delta)} u^2
+\int_{\mathbb B(p, \delta+\frac{C\delta}{10\beta})\backslash
\mathbb B(p, \delta)}|\nabla u|^2\geq c
\frac{\beta^2}{\delta^2}\int_{\mathbb B(p,
\delta+\frac{C\delta}{\beta})\backslash \mathbb B(p,
\delta+\frac{C\delta}{10\beta})} u^2.\label{imp}
\end{equation}
Let
$$ \frac{C\delta}{10\beta}\leq \lambda^{-1}.          $$
Since $u$ satisfies (\ref{star}), standard elliptic theory yields
that
\begin{equation}
|\nabla u(x)|^2\leq C(\frac{\beta}{\delta})^{n+2} \int_{y\in \mathbb
B(x, \frac{C\delta}{10\beta})} u^2(y)\,dy.
\end{equation}
We integrate last inequality for $x\in \mathbb B(p,
\delta+\frac{C\delta}{10\beta})\backslash \mathbb B(p, \delta)$. It
follows that
\begin{eqnarray}
\int_{\mathbb B(p, \delta+\frac{C\delta}{10\beta})\backslash \mathbb
B(p, \delta)} |\nabla u|^2  &\leq& C(\frac{\beta}{\delta})^{n+2}
\int_{\{x\in \mathbb B(p, \delta+\frac{C\delta}{10\beta})\backslash
\mathbb B(p, \delta), \ y\in \mathbb B(x, \frac{C\delta}{10\beta})\}
}
u^2(y)\,dy dx \nonumber \\
&\leq &C \frac{\beta^2}{\delta^2} \int_{y\in \mathbb B(p,
\delta+\frac{C\delta}{5\beta})\backslash \mathbb B(p,
\delta-\frac{C\delta}{10\beta})}u^2(y)\,dy,
\end{eqnarray}
where we have changed the order of integration in the last
inequlity. Substituting last inequality into (\ref{imp}) gives that
\begin{equation}
\int_{\mathbb B(p, \delta+\frac{C\delta}{5\beta})\backslash \mathbb
B(p, \delta-\frac{C\delta}{10\beta})} u^2\geq \int_{\mathbb B(p,
\delta+\frac{C\delta}{\beta})\backslash \mathbb B(p,
\delta+\frac{C\delta}{10\beta})} u^2.
\end{equation}
Recall that $ u(x)=e_\lambda(x) \exp\{\lambda \varrho(x)\}$. Let
$$  \varrho(x_0)=\max_{\mathbb B(p, \delta+\frac{C\delta}{5\beta})\backslash \mathbb
B(p, \delta-\frac{C\delta}{10\beta})}\varrho(x), \quad
\varrho(x_1)=\min_{\mathbb B(p,
\delta+\frac{C\delta}{\beta})\backslash \mathbb B(p,
\delta+\frac{C\delta}{10\beta})} \varrho(x).         $$ Then
$$ \lambda | \varrho(x_0)-\varrho(x_1)|\leq C\max_{\overline{\mathcal{M}}}|\nabla \varrho(x)| \delta,    $$
since $\beta\geq C_0\lambda$.
Furthermore,  thanks to the fact that $\nabla \varrho(x)$ is a bounded function in
$\overline{\mathcal{M}}$, we have
\begin{equation}
C\int_{\mathbb B(p, \delta+\frac{C\delta}{5\beta})\backslash \mathbb
B(p, \delta-\frac{C\delta}{10\beta})}e^2_\lambda\geq \int_{\mathbb
B(p, \delta+\frac{C\delta}{\beta})\backslash \mathbb B(p,
\delta+\frac{C\delta}{10\beta})}e^2_\lambda. \label{aim}
\end{equation}
Adding $\int_{\mathbb B(p,
\delta+\frac{C\delta}{10\beta})}e^2_\lambda$ to both sides of
(\ref{aim}) yields that
\begin{equation}
C\int_{\mathbb B(p, \delta+\frac{C\delta}{5\beta})}e^2_\lambda\geq
\int_{\mathbb B(p, \delta+\frac{C\delta}{\beta})}e^2_\lambda.
\end{equation}
If we replace  $\delta=\frac{\delta'}{1+\frac{C}{5\beta}}$, we get
\begin{equation}
C\int_{\mathbb B(p, \delta')}e^2_\lambda\geq \int_{\mathbb B(p,
\delta'+\frac{C\delta'}{\beta})}e^2_\lambda.
\end{equation}
Since we can choose $\beta=C_0\lambda$, by finite number of iteration, we arrive
at
\begin{equation}
\int_{\mathbb B(p, \delta)}e^2_\lambda\geq C\int_{\mathbb B(p,
\delta(1+\frac{1}{\lambda}))}e^2_\lambda. \label{sar}
\end{equation}
This completes conclusion (A) in Theorem \ref{th1}. Next we show the
$L^2$-Bernstein's inequality. By the standard elliptic estimates,
\begin{equation}
|\nabla e_\lambda(x)|^2\leq \frac{C}{r^{2+n}} \int_{\mathbb B(x,
r)}e^2_\lambda(y)\,dy \label{elli}
\end{equation}
if $\lambda r\leq 1$ and $\mathbb B(x,
r)\subset \mathcal{M}$ . Choosing $r=\frac{\delta}{\lambda}$ and
integrating over $x\in \mathbb B(p, \delta)$,
\begin{eqnarray}
\int_{\mathbb B(p, \delta)}|\nabla e_\lambda(x)|^2\,dx &\leq&
\frac{C}{r^{2+n}} \int_{\{y\in\mathbb B(x, r), \ x\in\mathbb B(p,
\delta)\}}e^2_\lambda(y) dy dx
\nonumber\\
&\leq &  \frac{C}{r^{2}}\int_{\mathbb B(p, \delta+r)}
e^2_\lambda(x)\, dx,
\end{eqnarray}
where we have changed the order of integration in last inequality.
Application of (\ref{sar}) yields that
\begin{equation}
\int_{\mathbb B(p, \delta)}|\nabla e_\lambda(x)|^2 \,dx\leq
\frac{C\lambda^2}{\delta^{2}} \int_{\mathbb B(p, \delta)}
e^2_\lambda(x)\,dx.
\end{equation}
Thus, we arrive at the conclusion (B).

We continue to obtain $L^\infty$ version of Bernstein's inequality.
For $x\in \mathbb B(p, \delta)$, choosing
$r=\frac{\delta}{\lambda}$, the refined doubling inequality
(\ref{sar}) and (\ref{elli}) yield that
\begin{eqnarray}
|\nabla e_\lambda(x)|^2 \leq \frac{C}{r^{2+n}} \int_{\mathbb B(x,
r)}e^2_\lambda &\leq &\frac{C}{r^{2+n}}\int_{\mathbb B(p,
\delta+r)}e^2_\lambda \nonumber \\
&\leq & \frac{C}{r^{2+n}}\int_{\mathbb B(p,
\delta)}e^2_\lambda\nonumber \\ &\leq&
\frac{C}{r^{2+n}}\delta^n\max_{\mathbb B(p, \delta) }e^2_\lambda.
\end{eqnarray}
Therefore,
\begin{equation}
|\nabla e_\lambda(x)|\leq
\frac{C\lambda^{\frac{n+2}{2}}}{\delta}\max_{\mathbb B(p, \delta)
}|e_\lambda|
\end{equation}
for any $x\in \mathbb B(p, \delta)$. The conclusion (C) in Theorem
\ref{th1} is arrived.

\end{proof}

\section{Upper bound of Steklov eigenfunctions}

In this section, we will prove the optimal upper bound for interior
Steklov eigenfunctions. Assume that $\mathcal{M}$ is a real analytic
Riemannian manifold with boundary. We first show the measure of
nodal sets in the neighborhood close to boundary, then show the upper
bound of nodal sets away from the boundary $\partial\mathcal{M}$.
Since $\mathcal{M}$ is a real analytic Riemannian manifold with
boundary, we may embed $\mathcal{M}\subset \mathcal{M}_1$ as a
relatively compact subset, where $\mathcal{M}_1$ is an open real
analytic Riemannian manifold. The real analytic Riemannian manifold
$\mathcal{M}$ and $\mathcal{M}_1$ are of the same dimension.  We analytically extend the eigenfunction $e_\lambda$ in $\mathcal{M}_1$. Denote the neighborhood of the boundary $\partial\mathcal{M}$ as $\mathcal{M}_r=\{ x\in \mathcal{M}_1| dist\{ x, \mathcal{M}\}\leq r\}$.
To do the analytic continuation across the boundary, we want to get rid of $\lambda$ on the boundary. We introduce the following lifting argument.
Let $$\hat{v}(x, t)=e^{\lambda t} e_\lambda(x).$$
Then $\hat{v}(x, t)$ satisfies the equation
\begin{equation}
\left \{ \begin{array}{rll}
\triangle_g \hat{v} +\partial_t^2 \hat{v}-\lambda^2 \hat{v}=0 \quad &\mbox{in} \ \mathcal{M}\times (-\infty, \ -\infty),  \medskip\\
\frac{\partial \hat{v}}{\partial \nu}-\frac{\partial \hat{v}}{\partial t}=0 \quad &\mbox{on} \ {\partial\mathcal{M}}\times (-\infty, \ -\infty).
\end{array}
\right.
\end{equation}
We also want to get rid of $\lambda$ in the equation.
Choose $${v}(x, t, s)=e^{i\lambda s} \hat{v}(x, t).$$
Then we get that
\begin{equation}
\left \{ \begin{array}{rll}
\triangle_g {v} +\partial_t^2 {v}+ \partial_s^2 {v}=0 \quad &\mbox{in} \ {\mathcal{M}}\times (-\infty, \ -\infty)\times (-\infty, \ -\infty),  \medskip\\
\frac{\partial {v}}{\partial \nu}-\frac{\partial {v}}{\partial t}=0 \quad &\mbox{on} \ {\partial\mathcal{M}}\times (-\infty, \ -\infty)\times (-\infty, \ -\infty).
\end{array}
\right.
\label{killz}
\end{equation}
We can see that (\ref{killz}) is as  a uniform elliptic equation with oblique boundary conditions. We
introduce the cubes with unequal radius as
\begin{align*}
\Omega_{R, \rho}=\{ (x, t, s)\in \mathbb R^{n+2}| |x_i|<R \ \mbox{when} \ i<n, |x_n|<\rho R, \ |t|<R, \ |s|<R\}
\end{align*}
and half-cube
\begin{align*}
\Omega^+_{R, \rho}=\{ (x, t, s)\in \mathbb R^{n+2} | |x_i|<R \ \mbox{when} \ i<n, 0\leq x_n<\rho R, \ |t|<R, \ |s|<R\}.
\end{align*}
Choose any point $p\in \partial\mathcal{M}$, using Fermi coordinates and rescaling arguments, we may consider the function $v(x, t, s )$  locally in the cube centered at origin with the flatten boundary. Hence, $v(x, t, s )$ satisfies the following equation locally
\begin{equation}
\left \{ \begin{array}{lll}
\triangle_g {v} +\partial_t^2 {v}+ \partial_s^2 {v}=0 \quad &\mbox{in} \ \Omega^+_{2, 1},  \medskip\\
\frac{\partial {v}}{\partial x_n}-\frac{\partial {v}}{\partial t}=0 \quad &\mbox{on} \  \Omega^+_{2, 1}\cap\{x_n=0\}.
\end{array}
\right.
\label{backgo}
\end{equation}

Thanks to the Cauchy-Kovalevsky theorem (Theorem 9.4.5)  in \cite{H}, we can extend $v(x, t, s)$ to the region $\Omega_{1, \rho}$, where $0<\rho$ depends only on $\mathcal{M}$. Furthermore, the growth of the extended $v$ is controlled as
\begin{align}
\|v\|_{L^\infty(\Omega_{{1}, \rho})}
\leq C\|v\|_{L^\infty (\Omega^+_{2, 1})},
\label{agree}
\end{align}
where $C$ depends only on $\mathcal{M}$.
By compactness of the manifold and  the uniqueness of the analytic continuation, it follows that
\begin{align}
-\triangle e_\lambda(x)= 0 \quad &\mbox{in} \ \widehat{\mathcal{M}}_1,
\label{fuckzel}
\end{align}
where $\widehat{\mathcal{M}}_1= \{ x\in \mathcal{M}_1| dist\{x, \ \mathcal{M}\}\leq \rho\}$.
Recall the definition of ${v}(x, t, s)=e^{\lambda t} e^{\lambda i s}  e_\lambda(x)$, it follows from (\ref{agree}) that
\begin{align}
\|e_\lambda\|_{L^\infty(\mathbb B_{\rho})}\leq e^{C\lambda} \|e_\lambda\|_{L^\infty(\mathbb B^+_2)}.
\end{align}

From the Proposition \ref{proo} and the even extension of $u$, the following doubling inequality holds in half balls as
$$
\|u\|_{L^\infty(\mathbb B^+_{2r})}\leq
e^{C\lambda}\|u\|_{L^\infty(\mathbb B^+_{r})}
$$
for $0<r<r_0$, where $r_0$ depends only $\mathcal{M}$.
From the relations of $u$ and $e_\lambda$, we obtain the following doubling inequality
\begin{equation}
\|e_\lambda \|_{L^\infty(\mathbb B^+_{2r})}\leq e^{C {\lambda}}\| e_\lambda \|_{L^\infty(\mathbb B^+_{r})}.
\label{halfdouble2}
\end{equation}

Iterating the doubling inequality (\ref{halfdouble2}) in the half balls by finite number of steps, we can show that
\begin{align}
\|e_\lambda\|_{L^\infty(\mathbb B_{\rho})}&\leq e^{C\lambda} \|e_\lambda\|_{L^\infty(\mathbb B^+_\frac{\rho}{2})} \nonumber \\
&\leq e^{C\lambda} \|e_\lambda\|_{L^\infty(\mathbb B_\frac{\rho}{2})}.
\end{align}
By rescaling arguments, it also implies that
\begin{align}
\|e_\lambda\|_{L^\infty(\mathbb B_{2r})}\leq e^{C\lambda} \|e_\lambda\|_{L^\infty(\mathbb B_r)}
\label{correct}
\end{align}
for any $r\leq \frac{\rho}{2}$ with $\mathbb B_{2r}\subset \widehat{\mathcal{M}}_1$ and $C$ depending only on $\mathcal{M}$.

To get the upper bounds of nodal sets for Steklov eigenfunctions,
we need to extend $e_\lambda(x)$ locally as a holomorphic function  in
$\mathbb{C}^n$. Applying elliptic estimates for $e_\lambda$ in (\ref{fuckzel}) in a ball $\mathbb B(p, r)
\subset \widehat{\mathcal{M}}_1$, we have
\begin{equation}
|\frac{ D^{{\alpha}} {e_\lambda}(p)}{{\alpha} !}|\leq
C^{|\alpha|}_1 r^{-|{\alpha}| }\|{e_\lambda}\|_{L^\infty},
\end{equation}
where ${\alpha}$ is a multi-index and $C_1>1$ depends on $\mathcal{M}$.
By translation, we still consider the point $p$ as the origin.
Summing a geometric series, we can extend $e_\lambda(x)$ to be a
holomorphic function $e_\lambda(z)$ with $z\in\mathbb{C}^n$ to have
\begin{equation}
\sup_{|z|\leq \frac{r}{2C_1}}|e_\lambda(z)|\leq C_2 \sup_{|x|\leq
r}|e_\lambda(x)|
\end{equation}
with $C_2>1$.
  Iterating the doubling inequality (\ref{correct}) finitely many times, by the rescaling arguments, we obtain that
\begin{equation}
\sup_{|z|\leq 2r}|e_\lambda(z)|\leq e^{C_3\lambda} \sup_{|x|\leq
r}|e_\lambda(x)| \label{dara}
\end{equation}
for $0<r<\rho_0$ with $\rho_0$ depending on $\mathcal{M}$ and $C_3$ depends on $\mathcal{M}$.

We need a lemma concerning the growth of a complex analytic function
with the number of zeros. See e.g. \cite{DF} and \cite{HL}.
\begin{lemma}
Suppose $f: \textbf{ B}(0, 1)\subset \mathbb{C}\to \mathbb{C}$ is an
analytic function satisfying
$$ f(0)=1\quad \mbox{and} \quad \sup_{ \textbf{B}(0, 1)}|f|\leq 2^N$$
for some positive constant $N$. Then for any $r\in (0, 1)$, there
holds
$$\sharp\{z\in\textbf{B}(0, r): f(z)=0\}\leq cN        $$
where $c$ depends on $r$. Especially, for $r=\frac{1}{2}$, there
holds
$$\sharp\{z\in \textbf{B}(0, 1/2): f(z)=0\}\leq N.        $$
\label{wwhy}
\end{lemma}
We are ready to show the upper bound of interior nodal sets of Steklov eigenfunctions based on doubling inequality and the growth control lemma, see the pioneering work by \cite{DF} and \cite{Lin}.
\begin{proof}[Proof of Theorem 2] We first prove the nodal sets in a
neighborhood $\mathcal{M}_{\frac{\rho}{4}}$.  By rescaling and
translation, we can argue on scales of order one. Let $p\in \mathbb
B_{1/4}$ be the point where the maximum of $|e_\lambda|$ in $\mathbb
B_{1/4}$ is attained. For each direction $\omega \in S^{n-1}$, set
$\hat{e}_{\omega}(z)=e_\lambda(p+z\omega)$ in $z\in \textbf{B}(0,
1)\subset\mathbb{C}$. Denoted by $N(\omega)=\sharp\{z\in\textbf{B}(0, 1/2)\subset
\mathcal{C}| \hat{e}_\omega(z)=0\}$. By the doubling property (\ref{dara}) and the
Lemma \ref{wwhy}, we have
\begin{eqnarray}
\sharp\{ x \in \mathbb B(p, 1/2) &|& x-p \ \mbox{is parallel to} \
\omega \ \mbox{and} \ e_\lambda(x)=0\} \nonumber\\&\leq&
\sharp\{z\in\textbf{B}(0, 1/2)\subset
\mathcal{C}| \hat{e}_\omega(z)=0\} \nonumber\\
&=&N(\omega)\leq C\lambda.
\end{eqnarray}
With aid of integral geometry estimates, it implies that
\begin{eqnarray}
H^{n-1}\{ x \in \mathbb B(p, 1/2)| e_\lambda(x)=0\} &\leq&
c(n)\int_{S^{n-1}} N(\omega)\,d \omega \nonumber\\
&\leq &\int_{S^{n-1}}C\lambda\, d\omega=C\lambda.
\end{eqnarray}
Therefore, we have 
\begin{eqnarray}
H^{n-1}\{ x \in \mathbb B(0, 1/4)| e_\lambda(x)=0\} \leq C\lambda.
\end{eqnarray}

By covering the compact manifold
$\mathcal{M}_{\frac{\rho}{4}}\subset \widehat{\mathcal{M}}_1$ by finite number
of coordinate charts, we arrive at
\begin{equation}H^{n-1}\{x\in
\mathcal{M}_{\frac{\rho}{4}}|e_\lambda(x)=0\}\leq C\lambda.
\label{last12}
\end{equation}

Next we deal with the measure of nodal sets in
$\mathcal{M}\backslash \mathcal{M}_{\frac{\rho}{4}}$. We have
obtained the doubling inequality (\ref{kaoni}) in the interior of the manifold.
Since $u(x)= e_\lambda(x) \exp\{\lambda \varrho(x)\}$ and $-\hat{C}
<\varrho(x)\leq \hat{C}$ for some constant $\hat{C}$ depending on
$\mathcal{M}$, it is true that
\begin{equation}
\|e_\lambda \|_{L^\infty(\mathbb B(p, \,{2r}))}\leq
e^{C\lambda}\|e_\lambda\|_{L^\infty(\mathbb B(p,\, {r}))}
\label{corre}
\end{equation}
holds for $p\in \mathcal{M}\backslash \mathcal{M}_{\frac{\rho}{4}}$
and $0<r\leq \rho_0\leq \frac{\rho}{4}$. We similarly extend
$e_\lambda(x)$ locally as a holomorphic function in $\mathbb {C}^n$.
Since $e_\lambda(x)$ is harmonic in $\mathcal{M}\backslash
\mathcal{M}_{\frac{\rho}{4}}$,  applying elliptic estimates in a
small ball $\mathbb B(p, \ r)$, we have
\begin{equation}
|\frac{ D^\alpha e_\lambda(p)}{\alpha !}|\leq
C^{|\alpha|}_4 r^{-|\alpha|}\|e_\lambda\|_{L^\infty},
\end{equation}
where $C_4>1$ depends only on $\mathcal{M}$.
We consider the point $p$ as the origin as well. Summing a geometric
series, we can extend $e_{\lambda}(x)$ to be a holomorphic function
$e_{\lambda}(z)$ with $z\in\mathbb{C}^n$. Moreover, we have
\begin{equation}
\sup_{|z|\leq \frac{r}{2C_4}}|e_{\lambda}(z)|\leq C_5 \sup_{|x|\leq
r}|e_{\lambda}(x)|
\end{equation}
with $C_5>1$.

Thanks to the doubling inequality (\ref{corre}), by finite steps of iterations, we obtain that
\begin{equation}
\sup_{|z|\leq \frac{r}{2C_4}} |e_\lambda(z)|\leq e^{C_{6}\lambda}
\sup_{|x|\leq \frac{r}{4C_4}}|e_\lambda(x)|
\end{equation}
with $C_6$ depends on $\mathcal{M}$. In
particular, by rescaling arguments,
\begin{equation}
\sup_{|z|\leq 2r}|e_\lambda(z)|\leq e^{C\lambda} \sup_{|x|\leq
r}|e_\lambda(x)| \label{dara1}
\end{equation}
holds for $0<r<\frac{\rho_0}{2}$ with ${\rho_0}$ depending on $\mathcal{M}$. Using the
same arguments as obtaining the nodal sets in the neighborhood of the
boundary, we take advantage of lemma \ref{wwhy} and the inequality
(\ref{dara1}).  By rescaling and translation, we can argue on scales
of order one. Let $p\in \mathbb B_{1/4}$ be the point where the
maximum of $|e_\lambda|$ in $\mathbb B_{1/4}$ is achieved. For each
direction $\omega \in S^{n-1}$, set $e^{\omega}_\lambda(z)=
e_\lambda(p+z\omega)$ in $z\in \textbf{B}(0, 1)\subset\mathbb{C}$.
From the doubling property (\ref{dara1}) and the lemma \ref{wwhy}
above, we have
\begin{eqnarray}
\sharp\{ x \in \mathbb B(p, 1/2) &|& x-p \ \mbox{is parallel to} \
\omega \ \mbox{and} \ e_\lambda(x)=0\} \nonumber\\&\leq&
\sharp\{z\in\textbf{B}(0, 1/2)\subset
\mathcal{C}| e_\lambda^\omega(z)=0\} \nonumber\\
&=&N(\omega)\leq C\lambda.
\end{eqnarray}
Thanks to the integral geometry estimates, we get
\begin{eqnarray}
H^{n-1}\{ x \in \mathbb B(p, 1/2)| e_\lambda(x)=0\} &\leq&
c(n)\int_{S^{n-1}} N(\omega)\,d \omega \nonumber\\
&\leq &\int_{S^{n-1}}C\lambda\, d\omega=C\lambda.
\end{eqnarray}
Thus, we obtain
\begin{eqnarray}
H^{n-1}\{ x \in \mathbb B(0, 1/4)| e_\lambda(x)=0\} \leq C\lambda.
\end{eqnarray}
Using the finite number of coordinate charts to cover the compact
manifold $\mathcal{M}\backslash\mathcal{M}_{\frac{\rho}{4}}$, we
obtain
\begin{equation} H^{n-1}\{x\in
\mathcal{M}\backslash\mathcal{M}_{\frac{\rho}{4}}|e_\lambda(x)=0\}\leq
C\lambda. \label{last2}
\end{equation}
Together with (\ref{last12}) and (\ref{last2}), we arrive at the
conclusion in Theorem \ref{th2}.

\end{proof}

\section{Appendix}
In this section, we provide the proof of Lemma \ref{carl} and  some arguments stated in the proof Proposition \ref{pro2}.
Recall that
$$ \|\mathcal{L}_\beta (v)\|^2_\phi\geq \frac{1}{2}\mathcal{A}-\mathcal{B},              $$
where
\begin{eqnarray}\mathcal
{L}_\beta (v)&=&\partial^2_t v+\big(2\beta \phi'+e^{2t} {\bar
b}_t+(n-2)+\partial_t \ln\sqrt{\gamma}\big)\partial_t v+ e^{2t} \bar
b_i\partial_i v \nonumber \\
&+&\big(\beta^2\phi'^2+\beta\phi'{\bar b}_t e^{2t}+\beta
\phi''+(n-2)\beta \phi'+\beta \partial_t \ln \sqrt{\gamma}\phi'\big)
v+\triangle_{\omega} v+ e^{2t} \bar q v
\end{eqnarray}
and
\begin{eqnarray} \mathcal{A}=\|
\partial^2_t v&+& \triangle_{\omega} v +\big(2\beta \phi'+e^{2t} {\bar
b}_t\big)\partial_t v+ e^{2t} \bar
b_i\partial_i v \nonumber \\
&+&\big(\beta^2\phi'^2+\beta\phi'{\bar b}_t e^{2t}+(n-2)\beta \phi'+
e^{2t} \bar q \big) v\|^2_{\phi} \end{eqnarray} and \begin{equation}
\mathcal{B}=\|\beta \phi'' v+\beta \partial_t \ln\sqrt{\gamma} \phi'
v+(n-2)\partial_t v+\partial_t \ln\sqrt{\gamma}\partial_t
v\|^2_{\phi}.
\end{equation}
Modifying the arguments in \cite{BC} and \cite{Zh}, we can obtain the following lemma, which verifies the proof of (\ref{four}) and (\ref{doit}) in Proposition \ref{pro2}.
\begin{lemma}
There holds that
\begin{align}
\|\mathcal{L}_\beta (v)\|^2_\phi& \geq
\frac{1}{4}\mathcal{A} \nonumber \\
& \geq C\beta^3\int |\phi''||v|^2
\phi'^{-3}\sqrt{\gamma}\,dt d \omega +C\beta \int |\phi''||D_\omega
v|^2
\phi'^{-3}\sqrt{\gamma}\,dt d \omega \nonumber \\
&+C\beta\int |\partial_t v|^2 \phi'^{-3}\sqrt{\gamma}\,dt d \omega.
\end{align}
\label{appen}
\end{lemma}
\begin{proof}
We decompose $\mathcal{A}$ as
$$ \mathcal{A}=\mathcal{A}'_1+\mathcal{A}'_2+\mathcal{A}'_3,      $$
where
$$\mathcal{A}'_1=\|\partial^2_t v
+\big(\beta^2\phi'^2+\beta\phi'{\bar b}_t e^{2t}+(n-2)\beta \phi'+
e^{2t} \bar q \big) v+ \triangle_{\omega} v\|^2_{\phi} $$
and
$$\mathcal{A}'_2=\|\big(2\beta \phi'+e^{2t} {\bar
b}_t\big)\partial_t v+e^{2t}\bar{b}_i\partial_i v\|^2_{\phi}
$$
and
\begin{eqnarray}\mathcal{A}'_3=2< \partial^2_t v+
\big(\beta^2\phi'^2+\beta\phi'{\bar b}_t e^{2t}+(n-2)\beta \phi'+
e^{2t} \bar q \big) v+\triangle_\omega v,\nonumber \\
\big(2\beta \phi'+e^{2t} {\bar b}_t\big)\partial_t
v+e^{2t}\bar{b}_i\partial_i v>_\phi. \nonumber
\end{eqnarray}
We first compute $\mathcal{A}'_1$. Let $\hat{\alpha}$ be some small positive constant. Recall that $\phi(t)=t-e^{\epsilon t}$. Since $|\phi^{''}|\leq 1$ and $\beta$ is large enough, it is true that
\begin{equation}
\mathcal{A}'_1\geq \frac{\hat{\alpha}}{\beta} \mathcal{A}^{''}_1,
\label{basic}
\end{equation}
where  $\mathcal{A}^{''}_1$ is given by
$$ \mathcal{A}^{''}_1=\left\Vert\sqrt{|\phi^{''}|}[\partial^2_t v
+\big(\beta^2\phi'^2+\beta\phi'{\bar b}_t e^{2t}+(n-2)\beta \phi'+
e^{2t} \bar q \big) v+ \triangle_{\omega} v ]\right\Vert_\phi^2. $$
We split $\mathcal{A}^{''}_1$ into three parts:
\begin{equation}
\mathcal{A}^{''}_1=\mathcal{K}_1+\mathcal{K}_2+\mathcal{K}_3,
\label{fuckgui}
\end{equation}
where
$$\mathcal{K}_1=\left\Vert \sqrt{|\phi^{''}|}\big(\partial^2_t v+\triangle_{\omega} v\big)\right\Vert_\phi^2 $$ and
$$\mathcal{K}_2=\left\Vert \sqrt{|\phi^{''}|}\big(\beta^2 \phi'^2+ \beta\phi'{\bar b}_t e^{2t}+(n-2)\beta \phi'+
e^{2t} \bar q \big) v \right\Vert_\phi^2 $$
and
$$\mathcal{K}_3=2\bigg< |\phi^{''}|(\partial^2_t v+\triangle_{\omega} v), \ \big(\beta^2 \phi'^2+ \beta\phi'{\bar b}_t e^{2t}+(n-2)\beta \phi'+
e^{2t} \bar q \big) v \bigg>_\phi.  $$
The expression $\mathcal{K}_1$ is considered to be a nonnegative term. We estimate $\mathcal{K}_2$. By the triangle inequality,
\begin{align}
\mathcal{K}_2\geq \beta^4 \left\Vert \sqrt{|\phi^{''}|} \phi' v\right\Vert_\phi^2- \left\Vert\sqrt{|\phi^{''}|}\big(\beta\phi'{\bar b}_t e^{2t}+(n-2)\beta \phi'+
e^{2t} \bar q \big) v   \right\Vert_\phi^2.
\label{faill}
\end{align}
Using the fact that $\beta>C(1+\|\bar b\|_{W^{1, \infty}}+\|\bar
q\|^{1/2}_{W^{1, \infty}})$, we have
\begin{align}
\left\Vert\sqrt{|\phi^{''}|}\big(\beta\phi'{\bar b}_t e^{2t}+(n-2)\beta \phi'+
e^{2t} \bar q \big) v   \right\Vert_\phi^2 &\leq C\beta^4  \left\Vert\sqrt{|\phi^{''}|} e^t v \right\Vert_\phi^2 \nonumber \\ &+
C\beta^2 \left\Vert \sqrt{|\phi^{''}|} v \right\Vert_\phi^2.
\label{fail}
\end{align}
Since $t$ is close to negative infinity and then $\phi'$ is close to $1$,  from (\ref{faill}) and (\ref{fail}), we obtain that
\begin{equation}
\mathcal{K}_2\geq C\beta^4 \left\Vert \sqrt{|\phi^{''}|} v \right\Vert_\phi^2,
\label{soso}
\end{equation}
where we also used the fact that $\phi'$ is close to $1$.
We derive a lower bound for $\mathcal{K}_3$. Integration by parts shows that
\begin{align}
\mathcal{K}_3&=-2\int |\phi^{''}||\partial_t v|^2 \big(\beta^2\phi'^2+ \beta\phi'{\bar b}_t e^{2t}+(n-2)\beta \phi'+
e^{2t} \bar q  \big) \phi'^{-3} \sqrt{\gamma} \ dtd\omega \nonumber \\
&-2 \int \partial_t v v \partial_t\big [|\phi^{''}|\big(\beta^2\phi'^2+\beta\phi'{\bar b}_t e^{2t}+(n-2)\beta \phi'+
e^{2t} \bar q    \big) \phi'^{-3} \sqrt{\gamma}   \big]\ dtd\omega \nonumber \\
&-2\int |\phi^{''}| |D_\omega v|^2 \big(\beta^2\phi'^2+\beta\phi'{\bar b}_t e^{2t}+(n-2)\beta \phi'+
e^{2t} \bar q \big) \phi'^{-3} \sqrt{\gamma} dtd\omega \nonumber \\
&-2 \int \beta |\phi^{''}|\phi'  \gamma^{ij}\partial_i v\partial_j{\bar b}_t e^{2t}  \phi'^{-3} \sqrt{\gamma} dtd\omega \nonumber \\
&-2 \int |\phi^{''}|  \gamma^{ij}\partial_i v\partial_j \bar{q}  e^{2t} v \phi'^{-3} \sqrt{\gamma} dtd\omega.
\end{align}
By the Cauchy-Schwartz inequality and the condition that $\beta>C(1+\|\bar b\|_{W^{1, \infty}}+\|\bar
q\|^{1/2}_{W^{1, \infty}})$, we arrive at
\begin{equation}
\mathcal{K}_3\geq -C\beta^2 \int |\phi^{''}|( |\partial_t v|^2+|D_\omega v|^2+v^2) \phi'^{-3} \sqrt{\gamma} dtd\omega.
\label{sogan}
\end{equation}
Since $\mathcal{K}_1$ is nonnegative, the combination of (\ref{fuckgui}), (\ref{soso}) and (\ref{sogan}) yields that
\begin{align}
\mathcal{A}^{''}_1&\geq C\beta^4 \left\Vert \sqrt{|\phi^{''}|} v \right\Vert_\phi^2- C\beta^2 \left\Vert \sqrt{|\phi^{''}|} \partial_t v \right\Vert_\phi^2 \nonumber \\
&-C\beta^2 \left\Vert \sqrt{|\phi^{''}|} |D_\omega v| \right\Vert_\phi^2.
\end{align}
From (\ref{basic}), it follows that
\begin{align}
\mathcal{A}^{'}_1&\geq C\hat{\alpha}\beta^3 \left\Vert \sqrt{|\phi^{''}|} v \right\Vert_\phi^2- C\hat{\alpha}\beta \left\Vert \sqrt{|\phi^{''}|} \partial_t v \right\Vert_\phi^2 \nonumber \\
&-C\hat{\alpha} \beta \left\Vert \sqrt{|\phi^{''}|} |D_\omega v| \right\Vert_\phi^2.
\label{dan1}
\end{align}
Recall that
$$\mathcal{A}'_2=\|\big(2\beta \phi'+e^{2t} {\bar
b}_t\big)\partial_t v+e^{2t}\bar{b}_i\partial_i v\|^2_{\phi}.
$$
By the triangle inequality, one has
$$\mathcal{A}'_2\geq 2\beta^2 \| \phi'\partial_t v\|^2_{\phi}- \|e^{2t} {\bar
b}_t \partial_t v+e^{2t}\bar{b}_i\partial_i v\|^2_{\phi}.
$$
It is obvious that
$$\mathcal{A}'_2\geq\frac{1}{\beta}\mathcal{A}'_2.   $$
From the assumption that $\beta>C(1+\|\bar b\|_{W^{1, \infty}}+\|\bar
q\|^{1/2}_{W^{1, \infty}})$, we obtain that
\begin{align}
\mathcal{A}'_2& \geq C\beta \| \phi'\partial_t v\|^2_{\phi}- C\beta \|e^{t}\partial_t v\|^2_{\phi}- C\beta \|e^{t}|D_\omega v|\|^2_{\phi}\nonumber \\
&\geq C\beta \| \phi'\partial_t v\|^2_{\phi}-C\beta \|e^{t}|D_\omega v|\|^2_{\phi}.
\label{dan2}
\end{align}
For the inner product $\mathcal{A}'_3$, using the arguments of integration by parts, since $e^t \ll 1$ as $t<T_0$ and $|T_0|$ is large enough, we can show a lower bound of $\mathcal{A}'_3$,
\begin{align}
\mathcal{A}'_3 &\geq C\beta \left\Vert \sqrt{|\phi^{''}|} |D_\omega v| \right\Vert_\phi^2- C\beta^3 \left\Vert e^t v\right\Vert_\phi^2
-C\beta \left\Vert \sqrt{|\phi^{''}|}\partial_t v \right \Vert_\phi^2 \nonumber \\
&-C\beta^2\left\Vert \sqrt{|\phi^{''}|}  v \right\Vert_\phi^2.
\label{dan3}
\end{align}
Recall that $ \mathcal{A}=\mathcal{A}'_1+\mathcal{A}'_2 +\mathcal{A}'_3 $. From (\ref{dan1}), (\ref{dan2}) and (\ref{dan3}), it follows that
\begin{align}
\mathcal{A}&\geq C\hat{\alpha} \beta^3 \int |\phi^{''}|v^2 \phi'^{-3} \sqrt{\gamma} dtd\omega+C\beta \int |\partial_t v|^2  \phi'^{-3} \sqrt{\gamma}
\nonumber \\& +
C\beta \int |\phi^{''}| |D_\omega v|^2 \phi'^{-3} \sqrt{\gamma} dtd\omega- C\beta^2 \int |\phi^{''}|  v^2 \phi'^{-3} \sqrt{\gamma} dtd\omega \nonumber \\&
-C\beta^3 \int e^{2t} v^2 \phi'^{-3} \sqrt{\gamma} dtd\omega- C\beta \int |\phi^{''}| |\partial_t v|^2 \phi'^{-3} \sqrt{\gamma} dtd\omega\nonumber \\
&-C\hat{\alpha}\beta \int |\phi^{''}| |D_\omega v|^2 \phi'^{-3} \sqrt{\gamma} dtd\omega- C\beta\int e^{2t} |D_\omega v|^2 \phi'^{-3} \sqrt{\gamma} dtd\omega.
\end{align}
If we choose $\hat{\alpha}$ to be appropriately small and take the fact $|\phi^{''}|>e^t$ into account, we obtain that
\begin{eqnarray}
C\mathcal{A}&\geq& \beta^3\int |\phi''||v|^2
\phi'^{-3}\sqrt{\gamma}\,dt d \omega +\beta \int |\phi''||D_\omega
v|^2
\phi'^{-3}\sqrt{\gamma}\,dt d \omega \nonumber \\
&&+\beta\int |\partial_t v|^2 \phi'^{-3}\sqrt{\gamma}\,dt d \omega.
\label{doit2}
\end{eqnarray}
Now we show $\mathcal{B}$ can be absorbed into $\mathcal{A}$ for large $|T_0|$ and large $\beta$. Since
$$|\partial_t \ln\sqrt{\gamma}|\leq C e^t\leq |\phi^{''}|,$$ then
\begin{align}
\mathcal{B}&=\|\beta \phi''v+\beta \partial_t \ln\sqrt{\gamma} \phi'
v+(n-2)\partial_t v+\partial_t \ln\sqrt{\gamma}\partial_t
v\|^2_{\phi} \nonumber \\
&\leq \beta^2 \int |\phi''| v^2 \phi'^{-3} \sqrt{\gamma} dt d \omega+ C\int |\partial_t v|^2 e^{2t} \phi'^{-3} \sqrt{\gamma} dt d \omega.
\label{kaogui}
\end{align}
Thus, the right hand side of (\ref{kaogui}) can be incorporated by the right hand side of (\ref{doit2}). Hence
the proof of the lemma is arrived.
\end{proof}
\begin{proof}[Proof of Lemma  \ref{carl}]
If we recall that $u=e^{-\beta\varphi(x)} v$, the proof of Lemma \ref{appen} just implies Lemma \ref{carl} stated in Section 2.
\end{proof}


\begin{thebibliography}{CL}

\bibitem[AKS]{AKS} N. Aronszajn, A. Krzywicki and J. Szarski, A unique continuation theorem for exterior differential forms on Riemannian manifolds, Arkiv f\"or Matematik 34(1963), 417-453.

\bibitem[BC]{BC}L. Bakri and J.B.  Casteras,  Quantitative uniqueness for Schr\"{o}dinger operator with regular
potentials, Math. Methods Appl. Sci. 37(2014), 1992-2008.

\bibitem[BL]{BL}K. Bellova and F.-H. Lin, Nodal sets of Steklov eigenfunctions,
Calc. Var. $\&$ PDE, 54(2015), 2239-2268.

\bibitem[Br]{Br} J. Br\"uning, \"Uber Knoten von Eigenfunktionen des
Laplace-Beltrami-Operators, Math. Z. 158(1978), 15-21.

\bibitem[ChM]{ChM} S. Chanillo and B.  Muckenhoupt, Nodal geometry on Riemannian manifolds, J. Differential Geom, 34(1991), no.1, 85-91. 

\bibitem[CM]{CM} T.H. Colding and W. P. Minicozzi II, Lower bounds for nodal sets
of eigenfunctions, Comm. Math. Phys. 306(2011), 777-784.

\bibitem[D]{D} R-T Dong, Nodal sets of eigenfunctions on Riemann surfaces, J. Differential Geom., 36(1992), 493--506.

\bibitem[DF]{DF}H. Donnelly and C. Fefferman, Nodal sets of eigenfunctions
on Riemannian manifolds, Invent. Math. 93(1988), 161-183.

 \bibitem[DF1]{DF1}
 H. Donnelly and C. Fefferman, Nodal sets for eigenfunctions of the Laplacian on surfaces, J. Amer.
Math. Soc., 3(1990), no. 2, 333--353.

\bibitem[DF2]{DF2} H. Donnelly and C. Fefferman, Nodal sets of eigenfunctions:
Riemannian manifolds with boundary, in: Analysis, Et Cetera,
Academic Press, Boston, MA, 1990, 251-262.

\bibitem[DF3]{DF3}H. Donnelly and C. Fefferman, Growth and geometry of eigenfunctions of the Laplacian. Analysis and partial differential equations,
635-655, Lecture notes in Pure and Appl. Math., 122, Dekker, New
York, 1990.

\bibitem[GR]{GR} B. Georgiev and G. Roy-Fortin, Polynomial upper bound on interior Steklov nodal sets, arXiv:1704.04484.

\bibitem[GP]{GP} A. Girouard and I. Polterovich,   Spectral geometry of the Steklov
problem, J. Spectral Theory 7(2017), no.2, 321-359.


\bibitem[Han]{Han} Q. Han, Nodal sets of harmonic functions, Pure Appl. Math. Q. 3(2007), no.3, part 2, 647-688. 

\bibitem[HL]{HL} Q. Han and F.-H. Lin, Nodal sets of solutions of Elliptic
Differential Equations, book in preparation (online at
http://www.nd.edu/qhan/nodal.pdf).

\bibitem[HLu]{HLu} X. Han and G. Lu, A geometric covering lemma and nodal sets of eigenfunctions, Math. Res. Lett., 18(2011), no. 2, 337-352.

\bibitem[HS]{HS} R. Hardt and L. Simon, Nodal sets for solutions of
ellipitc equations, J. Differential Geom. 30(1989), 505-522.

\bibitem[HSo]{HSo} H. Hezari and C.D. Sogge, A natural lower bound for the
size of nodal sets, Anal. PDE. 5(2012), no. 5, 1133-1137.

\bibitem[H]{H}L. H\"ormander, The analysis of linear partial differential operators I. Distribution theory and Fourier analysis, Reprint of the second (1990) edition, Classics in Mathematics. Springer-Verlag, Berlin, 2003.

\bibitem[Lin]{Lin}F.-H. Lin, Nodal sets of solutions of elliptic
equations of elliptic and parabolic equations, Comm. Pure Appl Math.
44(1991), 287-308.

\bibitem[Lo1]{Lo1} A. Logunov, Nodal sets of Laplace eigenfunctions: polynomial upper estimates of the Hausdorff measure,  Ann. of Math., 187(2018), 221--239.

\bibitem[Lo2]{Lo2} A. Logunov, Nodal sets of Laplace eigenfunctions: proof of Nadirashvili's conjecture and of the lower bound in Yau's conjecture, 187(2018), 241--262.
    
\bibitem[LM]{LM} A. Logunov and E. Malinnikova, Nodal sets of Laplace eigenfunctions: estimates of the Hausdorff measure in dimension two and three, 	50 years with Hardy spaces, 333-344, Oper. Theory Adv. Appl., 261, Birkhäuser/Springer, Cham, 2018.
    
\bibitem[Lu]{Lu} G. Lu, Covering lemmas and an application to nodal geometry on Riemannian manifolds, Proc. Amer. Math. Soc, 117(1993), no.4, 971-978.
    
\bibitem[M]{M} D. Mangoubi, A remark on recent lower bounds for nodal sets, Comm. Partial Differential Equations, 36(2011), no. 12, 2208--2212.

\bibitem[PST]{PST} I. Polterovich, D. Sher and J. Toth,
 Nodal length of Steklov eigenfunctions on real-analytic
Riemannian surfaces, J. reine angew. Math. 754(2019), 17-47. 

\bibitem[S]{S} S. Steinerberger, Lower bounds on nodal sets of eigenfunctions via the heat flow, Comm. Partial Differential Equations, 39(2014), no. 12, 2240--2261.

\bibitem[SWZ]{SWZ} C.D. Sogge, X. Wang and J. Zhu, Lower bounds for interior nodal sets of Steklov eigenfunctions,
 Proc. Amer. Math. Soc. 144(2016), no. 11, 4715-4722.

\bibitem[SZ]{SZ} C.D. Sogge and S. Zelditch, Lower bounds on the
Hausdorff measure of nodal sets, Math. Res. Lett. 18(2011), 25-37.

\bibitem[SZ1]{SZ1}C.D. Sogge and S. Zelditch,  Lower bounds on the Hausdorff measure of nodal sets II,
Math. Res. Lett. 19(2012), no.6, 1361-1364.

\bibitem[WZ]{WZ} X. Wang and J. Zhu,  A lower bound for the nodal sets of Steklov
eigenfunctions,  Math. Res. Lett. 22(2015), no.4, 1243-1253.

\bibitem[Z]{Z} S. Zelditch, Local and global analysis of eigenfunctions on Riemannian
manifolds, in: Handbook of Geometric Analysis, in: Adv. Lect. Math.
(ALM), vol. 7(1), Int. Press, Somerville, MA, 2008, 545-658.

\bibitem[Z1]{Z1} S. Zelditch, Measure of nodal sets of analytic
steklov eigenfunctions, Math. Res. Lett. 22(2015), no.6, 1821-1842.

\bibitem[Zh]{Zh}J. Zhu, Doubling property and vanishing order of Steklov
eigenfunctions, Comm. Partial Differential Equations 40(2015), no.
8, 1498-1520.

\bibitem[Zh1]{Zh1} J. Zhu, Interior nodal sets of Steklov eigenfunctions on surfaces,
 Anal. PDE 9(2016), no. 4, 859-880.

\end{thebibliography}
\end{document}